\documentclass[a4paper,12pt]{article}
\usepackage{amsmath}
\usepackage{amsthm}
\usepackage{amssymb}
\usepackage{enumerate}
\usepackage{url}
\usepackage[utf8]{inputenc}
\usepackage{comment}

\title{Stochastic analysis of Beckner's and related functional inequalities}

\author{Yuu Hariya\thanks{Supported in part by JSPS KAKENHI Grant Number 22K03330}}

\date{\empty}
\numberwithin{equation}{section}

\theoremstyle{plain}

\newtheorem{thm}{Theorem}[section]
\newtheorem{lem}[thm]{Lemma}
\newtheorem{prop}[thm]{Proposition}

\theoremstyle{definition}

\theoremstyle{remark}
\newtheorem{rem}[thm]{Remark}

\begin{document}

\def\N {\mathbb{N}}
\def\R {\mathbb{R}}
\def\Q {\mathbb{Q}}

\def\F {\mathcal{F}}

\def\al {\alpha }
\def\la {\lambda }
\def\ve {\varepsilon}
\def\Om {\Omega}

\def\ga {\gamma }

\newcommand\ND{\newcommand}
\newcommand\RD{\renewcommand}

\ND\lref[1]{Lemma~\ref{#1}}
\ND\tref[1]{Theorem~\ref{#1}}
\ND\pref[1]{Proposition~\ref{#1}}
\ND\sref[1]{Section~\ref{#1}}
\ND\ssref[1]{Subsection~\ref{#1}}
\ND\aref[1]{Appendix~\ref{#1}}
\ND\rref[1]{Remark~\ref{#1}} 
\ND\cref[1]{Corollary~\ref{#1}}
\ND\eref[1]{Example~\ref{#1}}
\ND\fref[1]{Fig.\ {#1} }
\ND\lsref[1]{Lemmas~\ref{#1}}
\ND\tsref[1]{Theorems~\ref{#1}}
\ND\dref[1]{Definition~\ref{#1}}
\ND\psref[1]{Propositions~\ref{#1}}
\ND\rsref[1]{Remarks~\ref{#1}}
\ND\sssref[1]{Subsections~\ref{#1}}

\ND\pr{\mathbb{P}}
\ND\ex{\mathbb{E}}
\ND\br{W}

\ND\gss[1]{\gamma_{#1}}
\ND\norm[2]{\left\|{#1}\right\|_{#2}}
\ND\lp[2]{L^{#1}({#2})}
\ND\calB{\mathcal{B}}
\ND\vp{\varphi}
\ND\U{U}
\ND\V{V}
\ND\dv{\theta}
\ND\D{d}
\ND\Ga{\Gamma}
\ND\ccp{C^{\infty }_{b,\mathrm{const}+}(\R ^{\D })}
\ND\ba{\mathbf{a}}
\ND\gd{\gamma_{\D}}
\ND\fp{F}
\ND\tu{\widetilde{U}}
\ND\nua{\nu}
\ND\nub{\nu '}
\ND\tv{\widetilde{V}}

\ND{\rmid}[1]{\mathrel{}\middle#1\mathrel{}}

\def\thefootnote{{}}

\maketitle 
\begin{abstract}
Beckner's inequality is a family of inequalities that interpolates the two fundamental 
functional inequalities, the logarithmic Sobolev and Poincar\'e's inequalities. It is 
parametrized by exponent $p\in (1,2]$ and it implies the logarithmic Sobolev inequality 
as $p\to 1$ and agrees with Poincar\'e's inequality when $p=2$. In this paper, employing a 
stochastic method, we prove an improvement of Beckner's inequality under the 
Gaussian measure when $4/3\le p<2$; in particular, when $p=3/2$, the error bound  
is expressed in terms of the entropy functional. A similar reasoning to the derivation 
of the improvement also enables us to obtain a H\"older-type inequality that holds 
among the entropy, variance and related functionals. 
\footnote{Mathematical Institute, Tohoku University, Aoba-ku, Sendai 980-8578, Japan}
\end{abstract}

\section{Introduction}\label{;intro}
Given a positive integer $\D $, let $\gss{\D }$ denote the standard Gaussian 
measure on $(\R ^{\D },\calB (\R ^{\D }))$ with $\calB(\R ^{\D })$ the Borel 
$\sigma $-field on $\R ^{\D }$: 
\begin{align*}
 \frac{d\gd }{dx}=\frac{1}{\sqrt{(2\pi )^{\D }}}\exp \!\left( 
 -\frac{|x|^{2}}{2}
 \right) ,\quad x\in \R ,
\end{align*}
where $dx$ is the Lebesgue measure on $(\R ^{d},\calB (\R ^{\D }))$.
Here and in the sequel, for every $x,y\in \R ^{\D }$, we write 
$x\cdot y$ for the usual inner product in $\R ^{\D }$ and 
$|x|$ for the Euclidean norm of $x$: $|x|=\sqrt{x\cdot x}$. 
For every $p>0$, let 
\begin{align*}
 \lp{p}{\gss{\D }}:=
 \left\{ 
 f\colon \R ^{\D }\to \R ;\,\text{$f$ is measurable and satisfies }
 \int _{\R ^{\D }}|f|^{p}\,d\gss{\D }<\infty 
 \right\} ,
\end{align*}
and set 
\begin{align*}
 \norm{f}{p}:=\left\{ 
 \int _{\R ^{\D }}|f|^{p}\,d\gss{\D }
 \right\} ^{1/p},\quad f\in \lp{p}{\gd }.
\end{align*}
(Although $\norm{\,\cdot \,}{p}$ is not a norm for $p<1$, we abuse the 
common notation in that case.) Let $C^{\infty }_{b}(\R ^{\D })$ be 
the class of real-valued bounded 
$C^{\infty }$-functions on $\R ^{\D }$ with bounded derivatives of 
any orders. In what follows, for simplicity, we restrict ourselves to 
the class $\ccp $ of functions $f$ such that each $f$ is given by 
the sum of a nonnegative function in $C^{\infty }_{b}(\R ^{\D })$ 
and a positive constant. For every $p>0$, we set the functional 
$I_{p}$ on $\ccp $ by 
\begin{align*}
 I_{p}(f):=
 \begin{cases}
 \displaystyle \frac{\norm{f}{p}^{p}-\norm{f}{1}^{p}}{p-1} & \text{for $p\neq 1$},\\[3mm]
 \displaystyle \int _{\R ^{\D }}f\log \frac{f}{\norm{f}{1}}\,d\gss{\D } & 
 \text{for $p=1$}.
 \end{cases}
\end{align*}
By Jensen's inequality, each $I_{p}$ is a nonnegative functional and 
it is readily seen that $I_{1}(f)=\lim _{p\to 1}I_{p}(f)$ for all $f\in \ccp $.
The quantity $I_{1}(f)$ is the (unnormalized) entropy of the measure 
$f\,dx$ relative to $\gd $ and $I_{2}(f)$ is the variance of $f$ under 
$\gd $: 
\begin{align*}
 I_{2}(f)=\int _{\R ^{\D }}\left( 
 f-\int _{\R ^{\D }}f\,d\gss{\D }
 \right) ^{2}d\gss{\D }.
\end{align*}
In those two cases when $p=1$ and $p=2$, it is known that 
the above functional satisfies 
\begin{align}\label{;lsip}
 I_{1}(f)\le \frac{1}{2}\norm{f^{-1}|\nabla f|^{2}}{1}, && 
 I_{2}(f)\le \norm{|\nabla f|}{2}^{2},
\end{align}
valid for any $f\in \ccp $. Here $\nabla f$ is the gradient of $f$. 
These two inequalities are the logarithmic Sobolev inequality and 
Poincar\'e's inequality, respectively. Beckner's inequality refers to a family 
of inequalities that interpolates the above two inequalities in such a way that, 
for $1<p\le 2$, 
\begin{align}\label{;beck}
 I_{p}(f)\le \frac{p}{2}\norm{f^{p-2}|\nabla f|^{2}}{1},\quad 
 f\in \ccp .
\end{align}
We note that the original formulation of Beckner's inequality \cite{bec} 
is different from \eqref{;beck} but it is readily seen that they are 
equivalent. For the logarithmic Sobolev, Poincar\'e's and Beckner's 
inequalities, we refer to a comprehensive monograph \cite{bgl}.

In this paper, by a stochastic argument, we first observe that, when 
$p=4/3$, Beckner's inequality \eqref{;beck} is improved as 
\begin{align}\label{;im43}
 I_{4/3}(f)+\frac{1}{3}I_{2/3}(f)^{2}\le \frac{2}{3}\norm{f^{-2/3}|\nabla f|^{2}}{1}; 
\end{align}
that is, it holds that 
\begin{align*}
 \norm{f}{4/3}^{4/3}-\norm{f}{1}^{4/3}
 +\left( 
 \norm{f}{1}^{2/3}-\norm{f}{2/3}^{2/3}
 \right) ^{2}\le \frac{2}{9}\norm{f^{-2/3}|\nabla f|^{2}}{1}
\end{align*}
for any $f\in \ccp $. We then ask a natural question of what about other 
values of $p$. In this paper, we give an answer to the question when 
$4/3<p<2$. We show that, for every $4/3<p<2$, there exists a 
nonnegative function $\fp \equiv \fp _{p}$ on $[0,\infty )\times [0,\infty )$, 
which is positive on $(0,\infty )\times [0,\infty )$, such that 
\begin{align}\label{;imp}
 I_{p}(f)+\frac{p(p-1)(2-p)}{4}\fp \!\left( 
 \frac{I_{2p-2}(f)}{p-1},\norm{f}{2p-2}
 \right) \le \frac{p}{2}\norm{f^{p-2}|\nabla f|^{2}}{1}
\end{align}
for any $f\in \ccp $. A precise description of our function 
$\fp (s,x),\,s,x\ge 0$, is given in \ssref{;sscp}; in particular, the function is 
increasing in $s$ and decreasing in $x$, and admits the 
following asymptotics for large values of $s$:
\begin{align}\label{;fasym}
 \lim _{s\to \infty }
 \frac{\fp (s,x)}{s^{\frac{p}{2(p-1)}}}
 =\frac{2^{\frac{3p-2}{2(p-1)}}(p-1)^{\frac{2-p}{p-1}}}{\sqrt{\pi }p(2-p)}
 \Gamma \!\left( 
 \frac{1}{2(p-1)}
 \right) ,
\end{align}
where $\Gamma \colon (0,\infty )\to (0,\infty )$ is the gamma function and 
$x\ge 0$ is arbitrary. Note that we may regard \eqref{;im43} as the case 
\begin{align}\label{;fp43}
 \fp (s,x)=\frac{1}{2}s^{2},\quad s,x\ge 0,
\end{align}
in \eqref{;imp}. When $p=3/2$, inequality~\eqref{;imp} reads 
\begin{align*}
 I_{3/2}(f)+\frac{3}{32}\fp \!\left( 
 2I_{1}(f),\norm{f}{1}
 \right) \le \frac{3}{4}\norm{f^{-1/2}|\nabla f|^{2}}{1},
\end{align*}
which connects $I_{3/2}$ with the entropy functional $I_{1}$ in the 
framework of Beckner's inequality; as far as we know, this is the 
first paper that reports such a phenomenon. For improvement of 
Beckner's inequality of different nature valid for all $1<p<2$, we refer to 
Ivanisvili--Volberg \cite{iv2}.

Our method used to derive the above improvement is stochastic, and a 
similar reasoning also enables us to obtain the following inequality of the 
H\"older type for functionals $I_{p}$, which is new to the best of our 
knowledge: for every pair $0<p<q$, it holds that 
\begin{align}\label{;hoel}
 I_{p}(f)\le \frac{\Gamma (1+1/q)}{\Gamma (1+1/q-p/q)}q^{p/q}I_{q}(f)^{p/q}
\end{align}
for any $f\in \ccp $. For example, if we apply the above inequality to the pair 
$(p,q)=(1,2)$, then we have 
\begin{align*}
 I_{1}(f)\le \sqrt{\frac{\pi }{2}}I_{2}(f)^{1/2},
\end{align*}
and hence from Poincar\'e's inequality in \eqref{;lsip}, we have the following 
upper bound on the entropy functional $I_{1}$: 
\begin{align*}
 I_{1}(f)\le \sqrt{\frac{\pi }{2}}\norm{|\nabla f|}{2}. 
\end{align*}
If we take $(p,q)=(1,3/2)$, then, appealing to 
inequality~\thetag{1.25} in \cite{iv1} by Ivanisvili--Volberg, we also 
see that 
\begin{align*}
 I_{1}(f)\le \frac{2^{5/3}\pi }{3^{5/6}\Gamma (1/3)}\norm{|\nabla f|}{3/2}
\end{align*}
by the fact that $\Gamma (2/3)\Gamma (1/3)=2\pi /\sqrt{3}$ 
(see \eqref{;g1} below). If we take $q=1$, then we have the following 
lower bound on $I_{1}$ for any $0<p<1$:
\begin{align*}
 \Gamma (1-p)^{1/p}\bigl( 
 \norm{f}{1}^{p}-\norm{f}{p}^{p}
 \bigr) ^{1/p}\le I_{1}(f).
\end{align*}
Additionally, we remark that the constant 
in front of $I_{q}(f)^{p/q}$ in \eqref{;hoel} is bounded for every fixed $p$; 
indeed, if we let $C(p,q)$ denote it, then 
\begin{align*}
 \lim _{q\downarrow p}C(p,q)=\lim _{q\to \infty }C(p,q)=1.
\end{align*}

We give an outline of the paper. In \sref{;spre}, we state a stochastic 
framework we will work in, and prove two preliminary lemmas that 
will be referred to throughout the paper. 
Inequalities~\eqref{;im43} and \eqref{;imp} are proven in 
\pref{;p43} and \tref{;tp}, respectively, in \sref{;sibi}, which will be 
closed with observation on the case $1<p<4/3$. 
We prove inequality~\eqref{;hoel} in \tref{;th} in \sref{;shii}. 
The paper is concluded with \sref{;scr}, in which we explain 
how the stochastic argument employed in this paper also applies 
to an improvement of Poincar\'e's inequality obtained by 
Goldstein--Nourdin--Peccati \cite{gnp}. 

\section{Preliminaries}\label{;spre}
 We adopt the notation used in \cite{har} by the author. 
Let the triplet $(\Om, \F ,\pr )$ be a probability space on which a 
$\D $-dimensional standard Brownian motion $\br =\{ \br _{t}\} _{0\le t\le 1}$ 
is defined so that $\pr \circ \br _{1}^{-1}=\gss{\D }$. Let $\{ \F _{t}\} _{0\le t\le 1}$ 
denote the (augmentation of) the natural filtration of $\br $. 
Given an arbitrary $f\in \ccp $, set the process 
$M=\{ M_{t}\} _{0\le t\le 1}$ by 
\begin{align*}
 M_{t}=\ex \!\left[ 
 f(\br _{1-t}+x)
 \right] \!\big| _{x=\br _{t}}. 
\end{align*}
Here $\ex $ stands for the expectation with respect to $\pr $. 
In what follows, unless otherwise specified, we fix $f$. 
By the Markov property of $\br $, each random variable $M_{t}$ is a version of 
the conditional expectation 
$
 \ex \left[ 
 f(\br _{1})\rmid|\F _{t}
 \right] 
$, 
and hence $M$ is an $\{ \F _{t}\} $-martingale; moreover, by the continuity and 
boundedness of $f$, it also determines a continuous process. Let 
an $\R ^{\D }$-valued process $\dv =\{ \dv _{t}=(\dv ^{i}_{t})_{i=1}^{\D }\} _{0\le t\le 1}$ 
be defined by 
\begin{align*}
 \dv ^{i}_{t}=\ex \!\left[ 
 \partial _{i}f(\br _{1-t}+x)
 \right] \!\big| _{x=\br _{t}},\quad i=1,\ldots ,\D .
\end{align*}
Here we have used the shorthand notation that  
\begin{align*}
 \partial _{i}f(x)=\frac{\partial f}{\partial x_{i}}(x)
\end{align*}
for $x=(x_{1},\ldots ,x_{\D })\in \R ^{\D }$. The same as $M$, every component 
of $\dv $ is a continuous $\{ \F _{t}\} $-martingale. 
Observe that, for every $i=1,\ldots ,\D $, 
by the boundedness of $\partial _{i}f$, 
\begin{align*}
 \ex \!\left[ 
 \partial _{i}f(\br _{1-t}+x)
 \right] &=\partial _{i}\ex \!\left[ 
 f(\br _{1-t}+x)
 \right] 
\end{align*}
for any $0\le t\le 1$ and $x\in \R ^{\D }$.
Then, by a simple application 
of It\^o's formula, the two processes $M$ and $\dv $ are 
related via 
\begin{align}
 M_{t}=\ex \!\left[ f(\br _{1})\right] +\int _{0}^{t}\dv _{s}\cdot d\br _{s},\quad 
 &0\le t\le 1, \label{;reldv}
\end{align}
$\pr $-a.s., which relation may also be deduced from the Clark--Ocone formula; 
see \cite[Section~2]{har} for details. Note that $M_{0}=\norm{f}{1}$ by the 
positivity of $f$. With relation~\eqref{;reldv} at disposal, the following lemma is also an 
immediate consequence of It\^o's formula: 

\begin{lem}\label{;lpre1}
Let $\vp \colon (0,\infty )\to \R $ be a $C^{2}$-function. Then, for every 
$0\le t\le 1$, it holds that 
\begin{align*}
 \ex [\vp (M_{t})]-\vp (M_{0})
 =\frac{1}{2}\int _{0}^{t}ds\,\ex \!\left[ 
 \vp ''(M_{s})|\dv _{s}|^{2}
 \right] .
\end{align*}
\end{lem}

\begin{proof}
By the fact that 
\begin{align*}
 M_{t}\ge \inf _{x\in \R ^{\D }}f(x)>0,\quad 0\le t\le 1,
\end{align*}
$\pr $-a.s., we may apply It\^o's formula to the process 
$\{ \vp (M_{t})\} _{0\le t\le 1}$ to get 
\begin{align*}
 \vp (M_{t})-\vp (M_{0})
 =\int _{0}^{t}\vp '(M_{s})\dv _{s}\cdot d\br _{s}
 +\frac{1}{2}\int _{0}^{t}\vp ''(M_{s})|\dv _{s}|^{2}\,ds,\quad 
 0\le t\le 1,
\end{align*}
$\pr $-a.s., in view of \eqref{;reldv}. Notice that, as $f\in \ccp $, 
the stochastic integral above 
gives rise to a true martingale because its integrand is a bounded process. 
Therefore, taking the expectation on both sides at time $t$, we have the claim. 
Here we have applied Fubini's theorem to the right-hand side thanks to the 
boundedness of the process $\{ \vp ''(M_{t})|\dv _{t}|^{2}\} _{0\le t\le 1}$.
\end{proof}

Let $\Phi $ be a $C^{4}$-function on $(0,\infty )$ such that 
\begin{align}\label{;acvx}
 \Phi ''(x)>0\quad \text{for all $x>0$}.
\end{align}
Thanks to the assumption on $f$, and as was seen in the proof of the 
previous lemma, both sides of the claimed equality in the lemma below 
also make sense; we do not repeat such a remark afterwards.

\begin{lem}\label{;lpre2}
Let $\Phi \colon (0,\infty )\to \R $ be as above. Then, for every $0\le s\le t\le 1$, 
it holds that 
\begin{equation}\label{;lpre2q}
\begin{split}
 &\ex \!\left[ 
 \Phi ''(M_{t})|\dv _{t}|^{2}
 \right] 
 -\ex \!\left[ 
 \Phi ''(M_{s})|\dv _{s}|^{2}
 \right] \\
 &=\ex \!\left[ 
 \frac{1}{\Phi ''(M_{t})}\left| 
 \Phi ''(M_{t})\dv _{t}-\Phi ''(M_{s})\dv _{s}
 \right| ^{2}
 \right] \\
 &\quad -\frac{1}{2}\ex \!\left[ 
 \Phi ''(M_{s})^{2}|\dv _{s}|^{2}\int _{s}^{t}du\,\left( 
 \frac{1}{\Phi ''}
 \right) ''\!(M_{u})|\dv _{u}|^{2}
 \right] .
\end{split}
\end{equation}
\end{lem}

\begin{proof}
First we verify that the left-hand side of the claimed equality is equal to 
\begin{align*}
 \ex \!\left[ 
 \frac{1}{\Phi ''(M_{t})}\left| 
 \Phi ''(M_{t})\dv _{t}-\Phi ''(M_{s})\dv _{s}
 \right| ^{2}
 \right] +\ex \!\left[ 
 \Phi ''(M_{s})|\dv _{s}|^{2}\left\{ 
 1-\frac{\Phi ''(M_{s})}{\Phi ''(M_{t})}
 \right\} 
 \right] .
\end{align*}
By developing the square in the first term, the above expression coincides with 
\begin{align*}
 \ex \!\left[ 
 \Phi ''(M_{t})|\dv _{t}|^{2}
 \right] 
 -2\ex \!\left[ 
 \Phi ''(M_{s})\dv _{t}\cdot \dv _{s}
 \right] 
 +\ex \!\left[ 
 \Phi ''(M_{s})|\dv _{s}|^{2}
 \right] .
\end{align*}
Since 
\begin{align*}
 \ex \!\left[ 
 \Phi ''(M_{s})\dv _{t}\cdot \dv _{s}
 \right] 
 &=\ex \!\left[ 
 \Phi ''(M_{s})\dv _{s}\cdot \ex \!\left[ 
 \dv _{t}\rmid| \F _{s}
 \right] 
 \right] \\
 &=\ex \!\left[ 
 \Phi ''(M_{s})|\dv _{s}|^{2}
 \right] 
\end{align*}
by the fact that each component of $\dv $ is a martingale, the conclusion 
follows.

It now remains to show that 
\begin{equation}\label{;plpre2q1}
\begin{split}
 &\ex \!\left[ 
 \Phi ''(M_{s})|\dv _{s}|^{2}\left\{ 
 1-\frac{\Phi ''(M_{s})}{\Phi ''(M_{t})}
 \right\} 
 \right] \\
 &=-\frac{1}{2}\ex \!\left[ 
 \Phi ''(M_{s})^{2}|\dv _{s}|^{2}\int _{s}^{t}du\left( 
 \frac{1}{\Phi ''}
 \right) ''\!(M_{u})|\dv _{u}|^{2}
 \right] .
\end{split}
\end{equation}
To this end, observe that we may replace in the left-hand side the 
reciprocal of $\Phi ''(M_{t})$ by 
\begin{align*}
 \ex \!\left[ 
 \frac{1}{\Phi ''(M_{t})}\rmid| \F _{s}
 \right] .
\end{align*}
Since 
\begin{align*}
 \frac{1}{\Phi ''(M_{t})}-\frac{1}{\Phi ''(M_{s})}
 =\int _{s}^{t}\left( \frac{1}{\Phi ''}
 \right) '\!(M_{u})\dv _{u}\cdot d\br _{u}
 +\frac{1}{2}\int _{s}^{t}\left( \frac{1}{\Phi ''}
 \right) ''\!(M_{u})|\dv _{u}|^{2}\,du,
\end{align*}
$\pr $-a.s., by It\^o's formula, the above conditional expectation is 
equal to 
\begin{align*}
 \frac{1}{\Phi ''(M_{s})}
 +\frac{1}{2}\ex \!\left[ 
 \int _{s}^{t}\left( \frac{1}{\Phi ''}
 \right) ''\!(M_{u})|\dv _{u}|^{2}\,du\rmid| \F_{s}
 \right] ,
\end{align*}
$\pr $-a.s., from which \eqref{;plpre2q1} follows readily.
\end{proof}

\begin{rem}\label{;rphi}
In addition to the convexity assumption~\eqref{;acvx} on $\Phi $, if we assume 
that $1/\Phi ''$ is concave on $(0,\infty )$, then the above lemma entails 
in particular that 
\begin{align*}
 \ex \!\left[ 
 \Phi ''(M_{t})|\dv _{t}|^{2}
 \right] 
 -\ex \!\left[ 
 \Phi ''(M_{s})|\dv _{s}|^{2}
 \right] \ge 0
\end{align*}
for all $0\le s\le t\le 1$; in fact, since the function 
\begin{align*}
 (0,\infty )\times \R ^{\D }\ni (x,y)\mapsto \Phi ''(x)|y|^{2}
\end{align*}
is convex by assumption, the process $\{ \Phi ''(M_{t})|\dv _{t}|^{2}\} _{0\le t\le 1}$ 
is a submartingale by the multidimensional conditional Jensen's inequality. 
Therefore, taking $t=1$ and integrating the left-hand side of the above inequality 
with respect to $s$ over $[0,1]$, we have 
\begin{align*}
 \int _{0}^{1}ds\,\ex \!\left[ 
 \Phi ''(M_{s})|\dv _{s}|^{2}
 \right] 
 \le \ex \!\left[ 
 \Phi ''(M_{1})|\dv _{1}|^{2}
 \right] .
\end{align*}
Combining the last inequality with \lref{;lpre1} yields 
\begin{align*}
 \ex \!\left[ 
 \Phi (M_{1})
 \right] -\Phi (M_{0})\le \frac{1}{2}\ex \!\left[ 
 \Phi ''(M_{1})|\dv _{1}|^{2}
 \right] ,
\end{align*}
namely, 
\begin{align}\label{;phi}
 \int _{\R ^{\D }}\Phi (f)\,d\gd -\Phi \!\left( 
 \int _{\R ^{\D }}f\,d\gd 
 \right) 
 \le \frac{1}{2}\int _{\R ^{\D }}\Phi ''(f)|\nabla f|^{2}\,d\gd .
\end{align}
The left-hand side is referred to as the \emph{$\Phi $-entropy} of $f$, and 
the above inequality is called the $\Phi $-entropy inequality. 
The logarithmic Sobolev inequality and Poincar\'e's inequality in \eqref{;lsip} 
correspond to the cases $\Phi (x)=x\log x$ and $\Phi (x)=x^{2}$, respectively, 
and Beckner's inequality~\eqref{;beck} to the case $\Phi (x)=x^{p}\ (1<p\le 2)$.
For $\Phi $-entropy inequalities, we refer to \cite[Section~7.6.1]{bgl}. 
We remark that the idea of proving the logarithmic Sobolev inequality stochastically 
in the Gaussian setting based on the Clark--Ocone formula goes back to \cite{chl}. 
We also remark that we have from \lref{;lpre2} the following error bound in 
\eqref{;phi}: 
\begin{align*}
 &\frac{1}{2}\int _{\R ^{\D }}\Phi ''(f)|\nabla f|^{2}\,d\gd 
 -\int _{\R ^{\D }}\Phi (f)\,d\gd +\Phi \!\left( 
 \int _{\R ^{\D }}f\,d\gd 
 \right) \\
 &\ge \frac{1}{2}\int _{0}^{1}ds\,
 \ex \!\left[ 
 \frac{1}{\Phi ''(M_{1})}\left| 
 \Phi ''(M_{1})\dv _{1}-\Phi ''(M_{s})\dv _{s}
 \right| ^{2}
 \right] ,
\end{align*}
which holds true for every $C^{2}$-function $\Phi $ fulfilling \eqref{;acvx} 
and the concavity assumption on $1/\Phi ''$ as observed in \cite[Section~3.2]{har}. 
The present paper is concerned with estimates on the second term on the 
right-hand side of \eqref{;lpre2q} for specific choices of $\Phi $; 
see the next section.
\end{rem}

\section{Improvement of Beckner's inequality}\label{;sibi}
 
We fix an arbitrary $f\in \ccp $ and keep the same notation as in \sref{;spre}. 
Let $1<p<2$. We may and do apply both of the two lemmas in the previous 
section to the function 
\begin{align*}
 x^{p},\quad x>0,
\end{align*}
with $t=1$, which in particular results in the inequality 
\begin{equation}\label{;rslti}
\begin{split}
 &\frac{\ex [M_{1}^{p}]-M_{0}^{p}}{p-1}
 +\frac{p(p-1)(2-p)}{4}\int _{0}^{1}dt\,\ex \!\left[ 
 M_{t}^{-p}|\dv _{t}|^{2}\int _{0}^{t}ds\,M_{s}^{2p-4}|\dv _{s}|^{2}
 \right] \\
 &\le \frac{p}{2}\ex [M_{1}^{p-2}|\dv _{1}|^{2}].
\end{split}
\end{equation}
Here the second term on the left-hand side follows from \lref{;lpre2} by observing 
that 
\begin{align*}
 &-\frac{1}{2}\int _{0}^{1}ds\,\ex \!\left[ 
 \Phi ''(M_{s})^{2}|\dv _{s}|^{2}\int _{s}^{1}du\left( 
 \frac{1}{\Phi ''}
 \right) ''\!(M_{u})|\dv _{u}|^{2}
 \right] \\
 &=\frac{p(p-1)^{2}(2-p)}{2}\int _{0}^{1}ds\,\ex \!\left[ 
 M_{s}^{2p-4}|\dv _{s}|^{2}\int _{s}^{1}du\,M_{u}^{-p}|\dv _{u}|^{2}
 \right] ,
\end{align*}
and then by Fubini's theorem.

\subsection{Case $p=4/3$}\label{;ssc43}

In this subsection, we prove 
\begin{prop}\label{;p43}
Inequality~\eqref{;im43} holds for any $f\in \ccp $.
\end{prop}

\begin{proof}
It suffices to prove that the second term on the left-hand side of \eqref{;rslti} 
is bounded from below by $(1/3)I_{2/3}(f)^{2}$ when $p=4/3$. To this end, observe that 
\begin{align*}
 -p=2p-4=-\frac{4}{3},
\end{align*}
and hence, by Fubini's theorem, 
\begin{align*}
 \int _{0}^{1}dt\,\ex \!\left[ 
 M_{t}^{-p}|\dv _{t}|^{2}\int _{0}^{t}ds\,M_{s}^{2p-4}|\dv _{s}|^{2}
 \right] 
 &=\frac{1}{2}\ex \!\left[ 
 \left( 
 \int _{0}^{1}dt\,M_{t}^{-4/3}|\dv _{t}|^{2}
 \right) ^{2}
 \right] \\
 &\ge \frac{1}{2}\left( 
 \int _{0}^{1}dt\,\ex \!\left[ 
 M_{t}^{-4/3}|\dv _{t}|^{2}
 \right] 
 \right) ^{2}.
\end{align*}
Here we have used Jensen's inequality and Fubini's theorem 
for the second line. By taking $\vp (x)=x^{2/3},\,x>0$, therein, 
we see from \lref{;lpre1} that 
\begin{align*}
 \int _{0}^{1}dt\,\ex [M_{t}^{-4/3}|\dv _{t}|^{2}]
 &=\frac{2}{ \frac{2}{3}(\frac{2}{3}-1)}\left( 
 \ex [M_{1}^{2/3}]-M_{0}^{2/3}
 \right) \\
 &=3I_{2/3}(f).
\end{align*}
Combining these proves that $(1/3)I_{2/3}(f)^{2}$ does not exceed the second term on 
the left-hand side of \eqref{;rslti} for $p=4/3$ and ends the proof of the proposition.
\end{proof}

\subsection{Case $4/3<p<2$}\label{;sscp}

We turn to the case $4/3<p<2$. Set the function 
$\phi \colon [0,\infty )\to (0,\infty )$ by 
\begin{align}\label{;defphi}
 \phi (v):=\int _{0}^{\infty }d\la \,\la e^{-\la }\exp \!\left( 
 -\frac{v}{2}\la ^{2}
 \right) ,\quad v\ge 0.
\end{align}
Notice that the function $v^{\frac{p-2}{2(p-1)}}\phi (v),\,v>0,$ 
is integrable over $(0,\infty )$ provided that $4/3<p<2$; indeed, a 
simple application of Fubini's theorem shows that 
\begin{align*}
 \int _{0}^{\infty }dv\,v^{\frac{p-2}{2(p-1)}}\phi (v)
 =2^{\frac{3p-4}{2(p-1)}}\Gamma \!\left( 
 \frac{3p-4}{2(p-1)}
 \right) \!\Gamma \!\left( 
 \frac{2-p}{p-1}
 \right) .
\end{align*}
For every real number $a$, we write $a_{+}$ for $\max \{ a,0\} $. 
Now we define $\fp \colon [0,\infty )\times [0,\infty )\to [0,\infty )$ by 
\begin{align}\label{;deffp}
 \fp (s,x):=C_{p}\int _{0}^{\infty }dv\left\{ 
 (s-x^{2p-2}v)_{+}
 \right\} ^{\frac{p}{2(p-1)}}v^{\frac{p-2}{2(p-1)}}\phi ((p-1)^{2}v),\quad 
 s,x\ge 0,
\end{align}
with $C_{p}$ the constant given by 
\begin{align}\label{;defcp}
 C_{p}=\frac{4(p-1)^{2}}{\pi p(2-p)}\sin \!\left( 
 \frac{\pi }{2}\cdot \frac{2-p}{p-1}
 \right) .
\end{align}
Note that 
\begin{align*}
 \frac{3p-4}{2(p-1)}+\frac{2-p}{2(p-1)}=1,
\end{align*}
and recall the following two relations satisfied by the gamma function 
(see, e.g., \cite[p.~3]{leb}):
\begin{align}
 \Gamma (z)\Gamma (1-z)&=\frac{\pi }{\sin \pi z},\quad 0<z<1, \label{;g1}\\
 2^{2z-1}\Gamma (z)\Gamma \!\left( 
 z+\frac{1}{2}
 \right) &=\sqrt{\pi }\Gamma (2z),\quad z>0. \notag
\end{align}
We apply those two relations with $z=(2-p)/\{ 2(p-1)\} \in (0,1)$ to find that 
\begin{align*}
 C_{p}\int _{0}^{\infty }dv\,v^{\frac{p-2}{2(p-1)}}\phi (v)
 =\frac{2^{\frac{3p-2}{2(p-1)}}(p-1)^{2}}{\sqrt{\pi }p(2-p)}
 \Gamma \!\left( 
 \frac{1}{2(p-1)}
 \right) ,
\end{align*}
which verifies that the function $\fp $ admits the asymptotics~\eqref{;fasym} 
by the dominated convergence theorem. We prove 
\begin{thm}\label{;tp}
For every $4/3<p<2$, inequality~\eqref{;imp} holds for any $f\in \ccp $ 
with $F$ defined by \eqref{;deffp}.
\end{thm}

In view of \eqref{;rslti}, the assertion of the theorem follows once we show the 
following proposition: 

\begin{prop}\label{;pp}
For every $4/3<p<2$, it holds that 
\begin{align}\label{;ppq}
 \int _{0}^{t}ds\,\ex \!\left[ 
 M_{s}^{-p}|\dv _{s}|^{2}\int _{0}^{s}du\,M_{u}^{2p-4}|\dv _{u}|^{2}
 \right] 
 =\ex \!\left[ 
 \fp \!\left( 
 \int _{0}^{t}ds\,M_{s}^{2p-4}|\dv _{s}|^{2},M_{t}
 \right) 
 \right] 
\end{align}
for all $0\le t\le 1$.
\end{prop}

The proof of \tref{;tp} is immediate from the above proposition.

\begin{proof}[Proof of \tref{;tp}]
Note that the function 
\begin{align*}
 (y_{+})^{\frac{p}{2(p-1)}},\quad y\in \R ,
\end{align*}
is convex since $p/\{ 2(p-1)\} >1$. Therefore, by the definition~\eqref{;deffp} of 
$\fp $ and Jensen's inequality, the right-hand side of \eqref{;ppq} is bounded 
from below by 
\begin{align*}
 \fp \!\left( 
 \ex \!\left[ 
 \int _{0}^{t}ds\,M_{s}^{2p-4}|\dv _{s}|^{2}
 \right] ,\left\{ \ex [M_{t}^{2p-2}]\right\} ^{\frac{1}{2p-2}}
 \right) .
\end{align*}
We put $t=1$ in \lref{;lpre1} and take $\vp (x)=x^{2p-2}$ for $p\neq 3/2$ and 
$\vp (x)=x\log x$ for $p=3/2$ to see that 
\begin{align}\label{;pppfq}
 \int _{0}^{1}ds\,\ex \!\left[ 
 M_{s}^{2p-4}|\dv _{s}|^{2}
 \right] 
 =\frac{I_{2p-2}(f)}{p-1}.
\end{align}
Consequently, the left-hand side of \eqref{;ppq} is not less than 
\begin{align*}
 \fp \!\left( 
 \frac{I_{2p-2}(f)}{p-1},\norm{f}{2p-2}
 \right) 
\end{align*}
when $t=1$. This proves the theorem owing to \eqref{;rslti}.
\end{proof}

The rest of this section is devoted to the proof of \pref{;pp}, which will be done 
through the two \lsref{;laltu} and \ref{;lpdefp}. Set the function 
$U\colon [0,\infty )\to [0,\infty )$ by 
\begin{align}\label{;defu}
 U(z)=\fp (z,1),\quad z\ge 0.
\end{align}
Observe that 
\begin{align}\label{;rfpu}
 \fp (s,x)=x^{p}U(x^{2-2p}s)
\end{align}
for any $s\ge 0$ and $x>0$. Put 
\begin{align}\label{;defnua}
 \nua :=-\frac{1}{2(p-1)}\in \left( 
 -\frac{3}{2},-\frac{1}{2}
 \right) .
\end{align}
We begin with an alternative integral representation of $U$.

\begin{lem}\label{;laltu}
We have 
\begin{align}\label{;laltuq}
 U(z)=C_{p}'\int _{0}^{z}du\,\int _{0}^{u}dv\,\int _{0}^{\infty }d\la \,\la ^{\nua +2}
 K_{\nua }(\la )\exp \!\left\{ 
 -\frac{(p-1)^{2}v}{2}\la ^{2}
 \right\} 
\end{align}
for all $z\ge 0$. Here the constant $C_{p}'$ is given by 
\begin{align*}
 C_{p}'=\frac{2^{\frac{1}{2(p-1)}}}{\sqrt{\pi }\Gamma \Bigl( 
 \frac{3p-4}{2(p-1)}
 \Bigr) }
\end{align*}
and $K_{\nua }$ is the modified Bessel function of the third kind 
(or Macdonald's function) of order $\nua $.
\end{lem}

For the modified Bessel functions, see, e.g., \cite[Section~5.7]{leb}.

Note that, thanks to Fubini's theorem, the above representation of 
$U$ is rewritten as 
\begin{align*}
 C_{p}'\int _{0}^{\infty }dv\,(z-v)_{+}\int _{0}^{\infty }d\la \,\la ^{\nua +2}
 K_{\nua }(\la )\exp \!\left\{ 
 -\frac{(p-1)^{2}v}{2}\la ^{2}
 \right\} 
\end{align*}
for every $z\ge 0$. The fact that the last expression agrees with 
\eqref{;deffp} with $(s,x)=(z,1)$ therein is of interest in its own right. 
Recall the fact (see, e.g., \cite[Equation~\thetag{5.16.4}]{leb}) that, for any 
$\mu \in \R \backslash \{ 0\} $,  
\begin{align*}
 \lim _{z\downarrow 0}z^{|\mu |}K_{\mu }(z)=2^{|\mu |-1}\Gamma (|\mu |).
\end{align*}
Therefore, since 
\begin{align}\label{;condnua}
 \nua +2-|\nua |=2-\frac{1}{p-1}>-1
\end{align}
for $p>4/3$, one sees that 
\begin{align*}
 \int _{0}^{\ve }d\la \,\la ^{\nua +2}K_{\nua }(\la )<\infty \quad \text{for any }\ve >0.
\end{align*}
In fact, we can say more. Recall the following integral representations of $K_{\mu }$ 
(see, e.g., \cite[Equation~\thetag{5.10.25}]{leb}) with $\mu $ an arbitrary real number:
\begin{align}
 K_{\mu }(z)&=\frac{1}{2}\left( 
 \frac{z}{2}
 \right) ^{\mu }\int _{0}^{\infty }\frac{du}{u}\,u^{-\mu }\exp \!\left( 
 -u-\frac{z^{2}}{4u}
 \right) \label{;reprk1}\\
 &=2^{\mu -1}z^{\mu }\int _{0}^{\infty }\frac{dv}{v}\,v^{\mu }\exp \!\left( 
 -\frac{1}{4v}-z^{2}v
 \right) \label{;reprk2}
\end{align}
for $z>0$, where the second representation is due to the change of the variables 
with $u=1/(4v)$ in the first. For any $\kappa ,\mu \in \R $ with $\kappa -|\mu |>-1$, 
by \eqref{;reprk1} and Fubini's theorem, 
\begin{equation}\label{;genrel}
\begin{split}
 \int _{0}^{\infty }d\la \,\la ^{\kappa }K_{\mu }(\la )
 &=\frac{1}{2^{\mu +1}}\int _{0}^{\infty }du\,u^{-\mu -1}e^{-u}
 \int _{0}^{\infty }d\la \,\la ^{\kappa +\mu }
 \exp \!\left( 
 -\frac{\la ^{2}}{4u}
 \right) \\
 &=2^{\kappa -1}\int _{0}^{\infty }du\,u^{(\kappa -\mu -1)/2}e^{-u}\int _{0}^{\infty }d\eta \,
 \eta ^{(\kappa +\mu -1)/2}e^{-\eta }\\
 &=2^{\kappa -1}\Gamma \!\left( 
 \frac{1+\kappa -\mu }{2}
 \right) \!\Gamma \!\left( 
 \frac{1+\kappa +\mu }{2}
 \right) ,
\end{split}
\end{equation}
where, for the second line, we changed the variables with $\la =2\sqrt{u\eta }$. 
Taking $\kappa =\nu +2$ and $\mu =\nua $ (which is allowed by 
\eqref{;condnua}) shows that 
\begin{align*}
 \int _{0}^{\infty }d\la \,\la ^{\nua +2}K_{\nua }(\la )
 =2^{\nua }\sqrt{\pi }\Gamma \!\left( 
 \nua +\frac{3}{2}
 \right) ,
\end{align*}
and hence 
\begin{align}\label{;cpd}
 C_{p}'\int _{0}^{\infty }d\la \,\la ^{\nua +2}K_{\nua }(\la )=1
\end{align}
by the definition of $\nua $.

\begin{proof}[Proof of \lref{;laltu}]
We write $\tu (z)$ for the right-hand side of claim~\eqref{;laltuq} for $z\ge 0$ and 
prove that the Laplace transform of the function $\tu $ coincides with that of $U$:
\begin{align}\label{;lts}
 \int _{0}^{\infty }dz\,e^{-\al z}\tu (z)
 =\int _{0}^{\infty }dz\,e^{-\al z}U(z)
\end{align}
for all $\al >0$. Fix $\al >0$ arbitrarily below. Noting that, by the integrability of 
the function $\la ^{\nua +2}K_{\nua }(\la ),\,\la >0$, as exhibited in \eqref{;cpd}, 
the value of $\tu (z)$ (resp.\ of $\tu '(z)$) grows at most quadratically (resp.\ linearly) 
as $z\to \infty $, we begin with the left-hand side of \eqref{;lts}, which is seen to be 
equal to 
\begin{align*}
 \frac{1}{\al ^{2}}\int _{0}^{\infty }dz\,e^{-\al z}\tu ''(z)
 &=\frac{C_{p}'}{\al ^{2}}\int _{0}^{\infty }d\la \,\frac{
 \la ^{\nua +2}K_{\nua }(\la )
 }{
 \al +\frac{(p-1)^{2}}{2}\la ^{2}
 }\\
 &=C_{p}'\al ^{(\nua -3)/2}\int _{0}^{\infty }d\la \,\frac{
 \la ^{\nua +2}K_{\nua }(\sqrt{\al }\la )
 }{
 1+\frac{(p-1)^{2}}{2}\la ^{2}
 },
\end{align*}
where the first line follows by Fubini's theorem and the second by a simple 
change of variables. We insert into the last expression 
\begin{align}\label{;reprk2d}
 K_{\nua }(\sqrt{\al }\la )
 =2^{\nua -1}\al ^{\nua /2}\la ^{-\nua }\int _{0}^{\infty }\frac{dv}{v}\,v^{\nua }
 \exp \!\left( 
 -\frac{\la ^{2}}{4v}-\al v
 \right) 
\end{align}
following from \eqref{;reprk2}, which results in 
\begin{align}
 &\int _{0}^{\infty }dz\,e^{-\al z}\tu (z) \notag \\
 &=2^{\nua -1}C_{p}'\al ^{\nua -3/2}\int _{0}^{\infty }dv\,
 e^{-\al v}v^{\nua -1}\int _{0}^{\infty }d\la \,
 \frac{\la ^{2}}{1+\frac{(p-1)^{2}}{2}\la ^{2}}
 \exp \!\left( 
 -\frac{\la ^{2}}{4v} 
 \right) \label{;expl}
\end{align}
again by Fubini's theorem. We turn to the right-hand side of \eqref{;lts}. 
By the definitions~\eqref{;defu}, \eqref{;deffp} and \eqref{;defnua} of 
$U$, $\fp $ and $\nua $, respectively, 
\begin{align*}
 \int _{0}^{\infty }dz\,e^{-\al z}U(z)
 &=\int _{0}^{\infty }dz\,e^{-\al z}F(z,1)\\
 &=C_{p}\int _{0}^{\infty }dz\,e^{-\al z}\int _{0}^{z}dv\,(z-v)^{-\nua +1/2}v^{\nua +1/2}
 \phi ((p-1)^{2}v)\\
 &=C_{p}\Gamma \!\left( 
 \frac{3}{2}-\nua 
 \right) \al ^{\nua -3/2}\int _{0}^{\infty }dv\,e^{-\al v}v^{\nua +1/2}\phi ((p-1)^{2}v).
\end{align*}
By comparing the last expression with \eqref{;expl}, it now remains to verify that 
\begin{align}\label{;inr}
 2^{\nua -1}C_{p}'\int _{0}^{\infty }d\la \,
 \frac{\la ^{2}}{1+\frac{(p-1)^{2}}{2}\la ^{2}}
 \exp \!\left( 
 -\frac{\la ^{2}}{4v} 
 \right) 
 =C_{p}\Gamma \!\left( 
 \frac{3}{2}-\nua 
 \right) v^{3/2}\phi ((p-1)^{2}v)
\end{align}
for all $v>0$. To this end, observe that 
\begin{align*}
 2^{\nua -1}C_{p}'=\frac{1}{2\sqrt{\pi }\Gamma (\nua +3/2)}
\end{align*}
(recall \eqref{;cpd}) and that 
\begin{align*}
 C_{p}\Gamma \!\left( 
 \frac{3}{2}-\nua 
 \right) 
 =\frac{1}{\Gamma (\nua +3/2)},
\end{align*}
which is because of the fact that, by the definition~\eqref{;defcp} of $C_{p}$, 
\begin{align*}
 C_{p}&=\frac{\sin \pi (-1/2-\nua )}{\pi (1/2-\nua )(-1/2-\nua )}\\
 &=\frac{1}{(1/2-\nua )(-1/2-\nua )\Gamma (-1/2-\nua )\Gamma (\nua +3/2)}\\
 &=\frac{1}{\Gamma (3/2-\nua )\Gamma (\nua +3/2)},
\end{align*}
where the second line follows from relation~\eqref{;g1}. Moreover, the integral in the left-hand side 
of \eqref{;inr} is equal to 
\begin{align*}
 v^{3/2}\int _{0}^{\infty }d\la \,\frac{\la ^{2}}{1+\frac{(p-1)^{2}v}{2}\la ^{2}}\exp \!\left( 
 -\frac{\la ^{2}}{4}
 \right) 
\end{align*}
by a change of variables. Consequently, the relation \eqref{;inr} in question is equivalently 
rephrased as 
\begin{align*}
 \frac{1}{2\sqrt{\pi }}\int _{0}^{\infty }d\la \,\frac{\la ^{2}}{1+\frac{v}{2}\la ^{2}}
 \exp \!\left( 
 -\frac{\la ^{2}}{4}
 \right) =\phi (v)
\end{align*}
for all $v>0$, or, by changing the variables with $\la =2\sqrt{\eta }$ and by the 
definition~\eqref{;defphi} of $\phi $, 
\begin{align}\label{;inrd}
 \frac{2}{\sqrt{\pi }}\int _{0}^{\infty }d\eta \,\frac{\sqrt{\eta }e^{-\eta }}{1+v\eta }
 =\int _{0}^{\infty }d\la \,\la e^{-\la }\exp \!\left( 
 -\frac{v}{4}\la ^{2}
 \right) 
\end{align}
for all $v>0$, where we have replaced $v$ by $v/2$. We give a proof of \eqref{;inrd} 
below for the reader's convenience.

Fix $v>0$ arbitrarily. By inserting the relation 
\begin{align*}
 \frac{1}{1+v\eta }=\int _{0}^{\infty }dy\,e^{-(1+v\eta )y}
\end{align*}
and making use of Fubini's theorem, the left-hand side of \eqref{;inrd} is rewritten as 
\begin{align}
 \frac{2}{\sqrt{\pi }}\int _{0}^{\infty }dy\,e^{-y}\int _{0}^{\infty }d\eta \,
 \sqrt{\eta }e^{-(1+vy)\eta }
 &=\int _{0}^{\infty }dy\,\frac{e^{-y}}{(1+vy)^{3/2}} \notag \\
 &=\frac{2}{v^{3/2}}e^{1/v}\int _{1/\sqrt{v}}^{\infty }\frac{dz}{z^{2}}\,e^{-z^{2}} \notag \\
 &=\frac{2}{v}-\frac{4}{v^{3/2}}e^{1/v}\int _{1/\sqrt{v}}^{\infty }dz\,e^{-z^{2}}, \label{;inrdl}
\end{align}
where, for the second line, we changed the variables with $1+vy=vz^{2}$, and used 
the integration by parts formula for the third. On the other hand, again by the 
integration by parts formula, the right-hand side of \eqref{;inrd} is equal to 
\begin{align*}
 \frac{2}{v}-\frac{2}{v}\int _{0}^{\infty }d\la \,\exp \!\left( 
 -\frac{v}{4}\la ^{2}-\la 
 \right) ,
\end{align*}
which is readily seen to agree with \eqref{;inrdl} by changing the variables with 
$(\sqrt{v}/2)(\la +2/v)=z$. The proof of the lemma is completed.
\end{proof}

\begin{lem}\label{;lpdefp}
The function $\fp $ satisfies the equation 
\begin{align}\label{;pdefp}
 \frac{\partial \fp }{\partial s}(s,x)x^{2p-4}+\frac{1}{2}\frac{\partial ^{2}\fp }{\partial x^{2}}(s,x)
 =x^{-p}s,\quad s,x>0,
\end{align}
with initial condition $\fp (0,x)=0$ for $x>0$.
\end{lem}

\begin{proof}
By the definition~\eqref{;deffp} of $\fp $, it is clear that $\fp $ satisfies the initial condition. 
Define the function $G\colon [0,\infty )\times [0,\infty )\to [0,\infty )$ through 
\begin{align*}
 \fp (s,x)=G\!\left( 
 s,\frac{x^{p-1}}{p-1}
 \right) ,\quad s,x\ge 0.
\end{align*}
In order to prove the lemma, it then suffices to show that $G$ satisfies the equation 
\begin{align}\label{;pdeg}
 \frac{\partial G}{\partial s}(s,y)+\frac{1}{2}\frac{\partial ^{2}G}{\partial y^{2}}(s,y)
 +\frac{2\nua +1}{2y}\frac{\partial G}{\partial y}(s,y)
 =C_{p}''y^{-2\nua -3}s,\quad 
 s,y>0,
\end{align}
with $C_{p}'':=(p-1)^{-2\nua -3}$.
To this end, by \eqref{;rfpu}, we see from \lref{;laltu} that $G$ admits the representation 
\begin{align*}
 G(s,y)=C_{p}'C_{p}''y^{-\nua }
 \int _{0}^{s}du\int _{0}^{u}dv\int _{0}^{\infty }d\la \,
 \la ^{\nua +2}K_{\nua }(y\la )\exp \!\left( 
 -\frac{v}{2}\la ^{2}
 \right) 
\end{align*}
for all $s\ge 0$ and $y>0$, owing to the repeated use of change of variables. 
For each $\la >0$, define 
\begin{align}\label{;defgla}
 g_{\la }(s,y):=y^{-\nua }K_{\nua }(\la y)\exp \!\left( 
 -\frac{\la ^{2}}{2}s
 \right) ,\quad s\ge 0,\ y>0.
\end{align}
Notice that every $g_{\la }$ solves the partial differential equation 
\begin{align}\label{;pdegla}
 \frac{\partial g}{\partial s}(s,y)+\frac{1}{2}\frac{\partial ^{2}g}{\partial y^{2}}(s,y)
 +\frac{2\nua +1}{2y}\frac{\partial g}{\partial y}(s,y)=0,\quad s,y>0;
\end{align}
indeed, from the ordinary differential equation (see \cite[Equation~\thetag{5.7.7}]{leb}) 
satisfied by $K_{\nua }$, we see that 
\begin{align}\label{;odek}
 \left( 
 \frac{d^{2}}{dz^{2}}+\frac{2\nua +1}{z}\frac{d}{dz}
 \right) \!\{ z^{-\nua }K_{\nua }(z)\} =z^{-\nua }K_{\nua }(z),\quad z>0,
\end{align}
from which the above partial differential equation follows readily. Therefore, 
for every $s,y>0$, if we verify that 
\begin{equation}\label{;verif}
\begin{split}
 \frac{\partial G}{\partial y}(s,y)
 &=C_{p}'C_{p}''\int _{0}^{s}du\int _{0}^{u}dv\int _{0}^{\infty }d\la \,\la ^{\nua +2}
 \frac{\partial g_{\la }}{\partial y}(v,y),\\
 \frac{\partial ^{2}G}{\partial y^{2}}(s,y)
 &=C_{p}'C_{p}''\int _{0}^{s}du\int _{0}^{u}dv\int _{0}^{\infty }d\la \,\la ^{\nua +2}
 \frac{\partial ^{2}g_{\la }}{\partial y^{2}}(v,y),
\end{split}
\end{equation}
then the left-hand side of equation~\eqref{;pdeg} turns into 
\begin{align*}
 &C_{p}'C_{p}''\int _{0}^{s}dv\int _{0}^{\infty }d\la \,\la ^{\nua +2}g_{\la }(v,y)
 -C_{p}'C_{p}''\int _{0}^{s}du\int _{0}^{u}dv\int _{0}^{\infty }d\la \,\la ^{\nua +2}
 \frac{\partial g_{\la }}{\partial v}(v,y)\\
 &=C_{p}'C_{p}''s\int _{0}^{\infty }d\la \,\la ^{\nua +2}g_{\la }(0,y)\\
 &=C_{p}'C_{p}''sy^{-2\nua -3}\int _{0}^{\infty }d\la \,\la ^{\nua +2}K_{\nua }(\la ),
\end{align*}
which agrees with the right-hand side of \eqref{;pdeg} thanks to \eqref{;cpd}. 
Here we have used Fubini's theorem for the second line and a change of 
variables for the third. 

It remains to verify \eqref{;verif}. Fix $s>0$ arbitrarily below. What to show 
is that, for every fixed $0<a<b$, the two functions $k_{1},k_{2}$ on 
$(0,\infty )$ defined respectively by 
\begin{align}\label{;k12}
 k_{1}(\la ):=\sup _{a<y<b}\left| 
 \frac{d}{dy}\{ y^{-\nua }K_{\nua }(\la y)\} 
 \right| ,\quad 
 k_{2}(\la ):=\sup _{a<y<b}\left| 
 \frac{d^{2}}{dy^{2}}\{ y^{-\nua }K_{\nua }(\la y)\} 
 \right| ,
\end{align}
for $\la >0$, are both integrable in the sense that 
\begin{align}\label{;int12}
 \int _{0}^{s}du\int _{0}^{u}dv\int _{0}^{\infty }d\la \,\la ^{\nua +2}k_{i}(\la )
 \exp \left( 
 -\frac{v}{2}\la ^{2}
 \right) <\infty ,\quad i=1,2.
\end{align}
Recall (see, e.g., \cite[Equation~\thetag{5.7.9}]{leb}) the relation 
\begin{align*}
 \frac{d}{dz}\{ z^{-\mu }K_{\mu }(z)\} =-z^{-\mu }K_{\mu +1}(z),\quad z>0,
\end{align*}
which holds for any $\mu \in \R $. The above relation entails in particular 
that the function $z^{-\mu }K_{\mu }(z),\,z>0$, is decreasing.  
Observing that 
\begin{align*}
 \frac{d}{dy}\{ y^{-\nua }K_{\nua }(\la y)\} 
 &=\la ^{\nua }\frac{d}{dy}\{ (\la y)^{-\nua }K_{\nua }(\la y)\} \\
 &=-\la ^{\nua +2}y(\la y)^{-\nua -1}K_{\nua +1}(\la y)
\end{align*}
for $y>0$, we see that $k_{1}(\la )$ is dominated by 
\begin{align*}
 \la ^{\nua +2}b(\la a)^{-\nua -1}K_{\nua +1}(\la a)
 =a^{-\nua -1}b\la K_{\nua +1}(a\la )
\end{align*}
for all $\la >0$, which verifies \eqref{;int12} for $i=1$ since we have 
\begin{align*}
 \int _{0}^{\infty }d\la \,\la ^{\nua +3}K_{\nua +1}(a\la )<\infty 
\end{align*}
in view of \eqref{;genrel}; indeed, $\nua $ fulfills 
$\nua +3-|\nua +1|>-1$. Turning to the case $i=2$, we see that  
\begin{align*}
 \frac{d^{2}}{dy^{2}}\{ y^{-\nua }K_{\nua }(\la y)\} 
 &=\la ^{\nua }\frac{d^{2}}{dy^{2}}\{ (\la y)^{-\nua }K_{\nua }(\la y)\} \\
 &=\la ^{\nua +2}\left\{ 
 (\la y)^{-\nua }K_{\nua }(\la y)+(2\nua +1)(\la y)^{-\nua -1}K_{\nua +1}(\la y)
 \right\} 
\end{align*}
for $y>0$, where the second line is due to \eqref{;odek}. Therefore, $k_{2}(\la )$ 
is dominated by 
\begin{align*}
 &\la ^{\nua +2}\left\{ 
 (\la a)^{-\nua }K_{\nua }(\la a)+|2\nua +1|(\la a)^{-\nua -1}K_{\nua +1}(\la a)
 \right\} \\
 &=a^{-\nua }\la ^{2}K_{\nua }(a\la )+|2\nua +1|a^{-\nua -1}\la K_{\nua +1}(a\la )
\end{align*}
for all $\la >0$, verifying \eqref{;int12} since 
\begin{align*}
 \int _{0}^{\infty }d\la \,\la ^{\nua +4}K_{\nua }(a\la )<\infty 
\end{align*}
as well. Therefore we have the lemma.
\end{proof}

We are in a position to prove \pref{;pp}.

\begin{proof}[Proof of \pref{;pp}]
We apply It\^o's formula to the process 
$\{ \fp (A_{t},M_{t})\} _{0\le t\le 1}$, where we set 
\begin{align*}
 A_{t}=\int _{0}^{t}M_{s}^{2p-4}|\dv _{s}|^{2}\,ds,\quad 0\le t\le 1.
\end{align*}
Then, in differential notation, by virtue of \eqref{;reldv}, 
\begin{align*}
 &d\fp (A_{t},M_{t})\\
 &=\frac{\partial \fp }{\partial s}(A_{t},M_{t})M_{t}^{2p-4}|\dv _{t}|^{2}\,dt
 +\frac{\partial \fp }{\partial x}(A_{t},M_{t})\dv _{t}\cdot d\br _{t}
 +\frac{1}{2}\frac{\partial ^{2}\fp }{\partial x^{2}}(A_{t},M_{t})|\dv _{t}|^{2}\,dt\\
 &=\frac{\partial \fp }{\partial x}(A_{t},M_{t})\dv _{t}\cdot d\br _{t}
 +M_{t}^{-p}|\dv _{t}|^{2}A_{t}\,dt
\end{align*}
with $\fp (A_{0},M_{0})=0$ a.s., where we have used \lref{;lpdefp} for the 
third line as well as for the initial value. The stochastic integral part 
in the above expression determines a true martingale thanks to the 
boundedness of the integrand, and hence, taking the expectation, 
we have the proposition.
\end{proof}

\subsection{On case $1<p<4/3$}\label{;sscop}
This section is intended for some observation on the case $1<p<4/3$. 
For every nonnegative integer $n$, set 
\begin{align*}
 H_{p,n}(t):=\int _{0}^{t}ds\,\ex \!\left[ 
 M_{s}^{2n-2-(2n-1)p}|\dv _{s}|^{2}
 \left( 
 \int _{0}^{s}du\,M_{u}^{2p-4}|\dv _{u}|^{2}
 \right) ^{n}
 \right] ,\quad 0\le t\le 1,
\end{align*}
where $p$ may be any real number.
The left-hand side of \eqref{;ppq} corresponds to the case $n=1$. 
Applying It\^o's formula to the process 
\begin{align*}
 M_{t}^{2n+2-(2n+1)p}\left( 
 \int _{0}^{t}ds\,M_{s}^{2p-4}|\dv _{s}|^{2}
 \right) ^{n+1},\quad 0\le t\le 1,
\end{align*}
we have the following relationship between $H_{p,n},\,n=0,1,2,\ldots $:
\begin{equation}\label{;relh}
\begin{split}
 H_{p,n}(t)
 &=\frac{1}{n+1}\ex \!\left[ 
 M_{t}^{2n+2-(2n+1)p}\left( 
 \int _{0}^{t}ds\,M_{s}^{2p-4}|\dv _{s}|^{2}
 \right) ^{n+1}
 \right] \\
 &\quad +\frac{2n+1}{2n+2}(p-1)\{ 2n+2-(2n+1)p\} 
 H_{p,n+1}(t),
\end{split}
\end{equation}
which holds true for any $t$ and $p$. We let $n=1$:
\begin{align}\label{;relh1}
 H_{p,1}(t)=\frac{1}{2}\ex \!\left[ 
 M_{t}^{4-3p}\left( 
 \int _{0}^{t}ds\,M_{s}^{2p-4}|\dv _{s}|^{2}
 \right) ^{2}
 \right] +\frac{3(p-1)(4-3p)}{4}H_{p,2}(t);
\end{align}
in particular, when $p=4/3$, the second term on the right-hand side 
vanishes, whence 
\begin{align*}
 H_{4/3,1}(t)=\frac{1}{2}\ex \!\left[ 
 \left( 
 \int _{0}^{t}ds\,M_{s}^{-4/3}|\dv _{s}|^{2}
 \right) ^{2}
 \right] ,
\end{align*}
as was observed in \ssref{;ssc43} in the case $t=1$. The last expression indicates 
that, when $p=4/3$, the function 
\begin{align*}
 \fp (s,x)=\frac{1}{2}s^{2},\quad s,x\ge 0,
\end{align*}
satisfies the partial differential equation~\eqref{;pdefp} in \lref{;lpdefp}, 
which is easily seen to be the case. We also recall \eqref{;fp43} in this 
respect. Returning to the case of general $p$, now we take $n=2$ in 
\eqref{;relh} and plug that relation into \eqref{;relh1} above to get 
\begin{align*}
 H_{p,1}(t)&=\frac{1}{2}\ex \!\left[ 
 M_{t}^{4-3p}\left( 
 \int _{0}^{t}ds\,M_{s}^{2p-4}|\dv _{s}|^{2}
 \right) ^{2}
 \right] \\
 &\quad +\frac{(p-1)(4-3p)}{4}
 \ex \!\left[ 
 M_{t}^{6-5p}\left( 
 \int _{0}^{t}ds\,M_{s}^{2p-4}|\dv _{s}|^{2}
 \right) ^{3}
 \right] \\
 &\quad +\frac{5(p-1)^{2}(4-3p)(6-5p)}{8}H_{p,3}(t);
\end{align*}
in particular, if we take $p=6/5$, then 
\begin{align}\label{;relh2}
 H_{6/5,1}(t)=\frac{1}{2}\ex \!\left[ 
 M_{t}^{2/5}\left( 
 \int _{0}^{t}ds\,M_{s}^{-8/5}|\dv _{s}|^{2}
 \right) ^{2}
 \right] +\frac{1}{50}
 \ex \!\left[ 
 \left( 
 \int _{0}^{t}ds\,M_{s}^{-8/5}|\dv _{s}|^{2}
 \right) ^{3}
 \right] ,
\end{align}
indicating that the function 
\begin{align*}
 \fp (s,x)=\frac{1}{2}x^{2/5}s^{2}+\frac{1}{50}s^{3},\quad s,x\ge 0,
\end{align*}
solves equation~\eqref{;pdefp}, which is indeed the case, too.  
Repeating the above procedure, we infer that, 
when $p=(2n+2)/(2n+1),\,n=1,2,\ldots $ (notice that index $\nua $ 
defined in the same way as in \eqref{;defnua} is a half-integer $-n-1/2$), 
the corresponding function $U$ defined through \eqref{;defu} from $\fp $ is a 
polynomial of order $n+1$; however, that information does not 
seem so useful in deriving a precise estimate. For instance, 
by dropping the first term on the right-hand side, we know 
from \eqref{;relh2} that 
\begin{align*}
 H_{6/5,1}(1)&\ge \frac{1}{50}\left( 
 \int _{0}^{1}ds\,\ex \!\left[ M_{s}^{-8/5}|\dv _{s}|^{2}\right] 
 \right) ^{3}\\
 &=\frac{5}{2}I_{2/5}(f)^{3},
\end{align*}
where the first line is due to Jensen's inequality and Fubini's 
theorem and the second to \lref{;lpre1}. For the left-most side above 
corresponds to the integral with respect to $t$ on the left-hand 
side of \eqref{;rslti} with $p=6/5$, we have 
\begin{align*}
 I_{6/5}(f)+\frac{3}{25}I_{2/5}(f)^{3}
 \le \frac{3}{5}\norm{f^{-4/5}|\nabla f|^{2}}{1}
\end{align*}
for any $f\in \ccp $, but of course, this improvement is far from 
being sharp as the above argument shows.  We do not pursue the case 
$1<p<4/3$ further in this paper.

\section{H\"older-type inequality for functionals $I_{p}$}\label{;shii}

The purpose of this section is to prove the following theorem: 

\begin{thm}\label{;th}
For every pair $0<p<q$, we have \eqref{;hoel} for any 
$f\in \ccp $.
\end{thm}

Given $q>0$, let 
\begin{align}\label{;defnub}
 \nub :=-\frac{1}{q}
\end{align} 
and fix $0<p<q$ below. To prove the theorem, define the function 
$J\colon [0,\infty )\times [0,\infty )\to [0,\infty )$ by 
\begin{equation}\label{;defj}
\begin{split}
 &J(s,x)\\
 &:=C_{p,\nub }
 \int _{0}^{\infty }dv\left\{ 
 (s-x^{q}v)_{+}
 \right\} ^{-p\nub }\!v^{\nub -1}\int _{0}^{\infty }d\la \,
 \frac{\la ^{2(p-1)\nub +1}}{1+\frac{q^{2}\la ^{2}}{8}}
 \exp \!\left( 
 -\frac{\la ^{2}}{4v}
 \right) 
\end{split}
\end{equation}
for $s,x\ge 0$, where the constant $C_{p,\nub }$ is given by 
\begin{align}\label{;cpnub}
 C_{p,\nub }
 =\frac{1}
 {2^{2(p-1)\nub +1}\Gamma (1-p\nub )\Gamma (1+p\nub )\Gamma (1+(p-1)\nub )}.
\end{align}
Then we have the following proposition: 

\begin{prop}\label{;ph}
It holds that, for every $0\le t\le 1$, 
\begin{align}\label{;phq}
 \int _{0}^{t}ds\,\ex \!\left[ 
 M_{s}^{p-2}|\dv _{s}|^{2}
 \right] 
 =\ex \!\left[ 
 J\!\left( 
 \int _{0}^{t}ds\,M_{s}^{q-2}|\dv _{s}|^{2},M_{t}
 \right) 
 \right] .
\end{align}
\end{prop}

Once we have the above proposition at our disposal, then the 
theorem follows from the following estimate on the function $J$: 

\begin{lem}\label{;lh}
We have 
\begin{align}\label{;lhq}
 J(s,x)\le \frac{2\Gamma (1+1/q)}{p\Gamma (1+1/q-p/q)}
 \!\left( 
 \frac{q^{2}}{2}
 \right) ^{p/q}\!s^{p/q}
\end{align}
for all $s,x\ge 0$.
\end{lem}

Deferring the proofs of \pref{;ph} and \lref{;lh} to \sssref{;ssprfph} and 
\ref{;ssprflh}, respectively, we prove \tref{;th}.

\begin{proof}[Proof of \tref{;th}]
Since $0<p/q<1$ by assumption, we see that the right-hand side of \eqref{;phq}
is dominated by 
\begin{align*}
 \frac{2\Gamma (1+1/q)}{p\Gamma (1+1/q-p/q)}
 \!\left( 
 \frac{q^{2}}{2}
 \right) ^{p/q}
 \!\left( 
 \int _{0}^{t}ds\,\ex \!\left[ 
 M_{s}^{q-2}|\dv _{s}|^{2}
 \right] 
 \right) ^{p/q}
\end{align*}
by \lref{;lh} and Jensen's inequality together with Fubini's theorem. 
We put $t=1$. In the same way as in the verification of \eqref{;pppfq}, 
we have 
\begin{align*}
 \int _{0}^{1}ds\,\ex \!\left[ 
 M_{s}^{q-2}|\dv _{s}|^{2}
 \right] =\frac{2}{q}I_{q}(f)
\end{align*}
by \lref{;lpre1}, which holds true with $q$ replaced by $p$ as well. Combining 
these, we reach the conclusion.
\end{proof}

\subsection{Proof of \pref{;ph}}\label{;ssprfph}

This subsection is devoted to the proof of \pref{;ph}, which proceeds 
almost identically to that of \pref{;pp}.

Define $V\colon [0,\infty )\to [0,\infty )$ by 
\begin{align}\label{;defv}
 V(z):=J(z,1),\quad z\ge 0.
\end{align}
Then, $J$ is expressed as 
\begin{align}\label{;rjv}
 J(s,x)=x^{p}V(s/x^{q})
\end{align}
for any $s\ge 0$ and $x>0$. We shall see that, as does the function $U$ in \ssref{;sscp}, 
the function $V$ admits an alternative integral representation in terms of 
the modified Bessel function of the third kind.

\begin{lem}\label{;laltv}
We have 
\begin{align}\label{;laltvq}
 V(z)=C_{p,\nub }'\int _{0}^{z}du\,\int _{0}^{\infty }d\la \,
 \la ^{(2p-1)\nub +1}K_{\nub }(\la )\exp \!\left( 
 -\frac{q^{2}u}{8}\la ^{2}
 \right) 
\end{align}
for all $z\ge 0$. Here the constant $C_{p,\nub }'$ is given by 
\begin{align}\label{;cpnubd}
 C_{p,\nub }'=\frac{1}{
 2^{(2p-1)\nub }\Gamma (1+p\nub )\Gamma (1+(p-1)\nub )
 }.
\end{align}
\end{lem}

Since $p/q<1$, we are allowed to take $\kappa =(2p-1)\nub +1$ and 
$\mu =\nub $ in \eqref{;genrel} to see that 
\begin{align}\label{;cpnubdd}
 C_{p,\nub }'\int _{0}^{\infty }d\la \,\la ^{(2p-1)\nub +1}K_{\nub }(\la )=1.
\end{align}

\begin{proof}[Proof of \lref{;laltv}]
We write $\tv (z)$ for the right-hand side of \eqref{;laltvq} for $z\ge 0$ 
and aim to show that, for every fixed $\al >0$, 
\begin{align}\label{;ltss}
 \int _{0}^{\infty }dz\,e^{-\al z}\tv (z)=\int _{0}^{\infty }dz\,e^{-\al z}V(z).
\end{align}
The left-hand side is equal to 
\begin{align*}
 \frac{1}{\al }\int _{0}^{\infty }dz\,e^{-\al z}\tv '(z)
 &=\frac{C_{p,\nub }'}{\al }\int _{0}^{\infty }d\la \,
 \frac{\la ^{(2p-1)\nub +1}K_{\nub }(\la )}{
 \al +\frac{q^{2}}{8}\la ^{2}
 }\\
 &=C_{p,\nub }'\al ^{(p-1/2)\nub -1}
 \int _{0}^{\infty }d\la \,
 \frac{\la ^{(2p-1)\nub +1}K_{\nub }(\sqrt{\al }\la )}{
 1+\frac{q^{2}}{8}\la ^{2}
 },
\end{align*}
where the first line is due to Fubini's theorem and the second to 
a change of variables. Into the last expression, we plug \eqref{;reprk2d} 
with $\nua $ therein replaced by $\nub $ to obtain 
\begin{align*}
 &\int _{0}^{\infty }dz\,e^{-\al z}\tv (z)\\
 &=2^{\nub -1}C_{p,\nub }'\al ^{p\nub -1}\int _{0}^{\infty }dv\,
 e^{-\al v}v^{\nub -1}\int _{0}^{\infty }d\la \,
 \frac{\la ^{2(p-1)\nub +1}}{
 1+\frac{q^{2}}{8}\la ^{2}
 }\exp \!\left( 
 -\frac{\la ^{2}}{4v}
 \right) 
\end{align*}
by Fubini's theorem. Note that, in the last expression,  
\begin{align*}
 \al ^{p\nub -1}
 =\frac{1}{\Gamma (1-p\nub )}\int _{0}^{\infty }dz\,e^{-\al z}z^{-p\nub }.
\end{align*}
In view of the definitions~\eqref{;defv} and \eqref{;defj} of $V$ and $J$, 
respectively, claim~\eqref{;ltss} follows by seeing that 
\begin{align*}
 \frac{2^{\nub -1}C_{p,\nub }'}{\Gamma (1-p\nub )}=C_{p,\nub },
\end{align*}
which is the case by the definitions~\eqref{;cpnubd} and \eqref{;cpnub} 
of $C_{p,\nub }'$ and $C_{p,\nub }$, respectively. Therefore the lemma 
is proven.
\end{proof}

\begin{rem}\label{;rbdv}
Relation~\eqref{;cpnubdd} together with \lref{;laltv} entails that 
\begin{align}\label{;bdv}
 V(z)\le z\quad \text{for all $z\ge 0$}.
\end{align}
We will return to the above fact at the end of this section in \rref{;rcons}.
\end{rem}

By the above lemma, we are able to prove a similar assertion to \lref{;lpdefp} 
as to the function $J$.

\begin{lem}\label{;lpdej}
The function $J$ satisfies the equation 
\begin{align}\label{;pdej}
 \frac{\partial J}{\partial s}(s,x)x^{q-2}+\frac{1}{2}\frac{\partial ^{2}J}{\partial x^{2}}(s,x)
 =x^{p-2},\quad s,x>0,
\end{align}
with initial condition $J(0,x)=0$ for $x>0$.
\end{lem}

\begin{proof}
By definition~\eqref{;defj}, it is clear that $J$ fulfills the initial condition. 
In order to prove \eqref{;pdej}, by defining the function 
$G\colon [0,\infty )\times [0,\infty )\to [0,\infty )$ through 
\begin{align*}
 J(s,x)=G\!\left( 
 s,\frac{x^{q/2}}{q/2}
 \right) ,\quad s,x\ge 0,
\end{align*}
it suffices to show that $G$ satisfies the equation 
\begin{align}\label{;pdegd}
 \frac{\partial G}{\partial s}(s,y)+\frac{1}{2}\frac{\partial ^{2}G}{\partial y^{2}}(s,y)
 +\frac{2\nub +1}{2y}\frac{\partial G}{\partial y}(s,y)
 =C_{p,\nub }''y^{-2p\nub -2},\quad 
 s,y>0,
\end{align}
with $C_{p,\nub }'':=(2|\nub |)^{2p\nub +2}$. By noting \eqref{;rjv}, and 
by the repeated use of change of variables, we know from 
\lref{;laltv} that $G$ is expressed as 
\begin{align*}
 G(s,y)=C_{p,\nub }'C_{p,\nub }''y^{-\nub }
 \int _{0}^{s}du\int _{0}^{\infty }d\la \,
 \la ^{(2p-1)\nub +1}K_{\nub }(y\la )\exp \!\left( 
 -\frac{u}{2}\la ^{2}
 \right) 
\end{align*}
for all $s\ge 0$ and $y>0$. For each $\la >0$, now we define $g_{\la }$ by 
\eqref{;defgla} with $\nua $ therein replaced by $\nub $. For every $s,y>0$, 
if we prove that 
\begin{align*}
 \frac{\partial G}{\partial y}(s,y)
 &=C_{p,\nub }'C_{p,\nub }''\int _{0}^{s}du\int _{0}^{\infty }d\la \,\la ^{(2p-1)\nub +1}
 \frac{\partial g_{\la }}{\partial y}(u,y),\\
 \frac{\partial ^{2}G}{\partial y^{2}}(s,y)
 &=C_{p,\nub }'C_{p,\nub }''\int _{0}^{s}du\int _{0}^{\infty }d\la \,\la ^{(2p-1)\nub +1}
 \frac{\partial ^{2}g_{\la }}{\partial y^{2}}(u,y),
\end{align*}
then, replacing $\nua $ by $\nub $ in \eqref{;pdegla}, we see that the left-hand side 
of \eqref{;pdegd} agrees with 
\begin{align*}
 &C_{p,\nub }'C_{p,\nub }''\int _{0}^{\infty }d\la \,\la ^{(2p-1)\nub +1}g_{\la }(s,y)
 -C_{p,\nub }'C_{p,\nub }''\int _{0}^{s}du\int _{0}^{\infty }d\la \,\la ^{(2p-1)\nub +1}
 \frac{\partial g_{\la }}{\partial u}(u,y)\\
 &=C_{p,\nub }'C_{p,\nub }''\int _{0}^{\infty }d\la \,\la ^{(2p-1)\nub +1}g_{\la }(0,y)\\
 &=C_{p,\nub }'C_{p,\nub }''y^{-2p\nub -2}\!\int _{0}^{\infty }d\la \,\la ^{(2p-1)\nub +1}K_{\nub }(\la ),
\end{align*}
which is the right-hand side of \eqref{;pdegd} by relation~\eqref{;cpnubdd}. Here we have 
used Fubini's theorem for the second line and made a change of variables for the third. 
Now we repeat the same argument as in the proof of \lref{;lpdefp} to see that the 
proof is reduced to showing that, for every fixed $s>0$ and $0<a<b$, 
\begin{align}\label{;int12d}
 \int _{0}^{s}du\int _{0}^{\infty }d\la \,\la ^{(2p-1)\nub +1}k_{i}(\la )
 \exp \left( 
 -\frac{u}{2}\la ^{2}
 \right) <\infty ,\quad i=1,2,
\end{align}
where $k_{1}$ and $k_{2}$ are defined in the same way as in \eqref{;k12} with 
$\nub $ replacing $\nua $ therein. Claim~\eqref{;int12d} follows once we show 
both of 
\begin{align*}
 \int _{0}^{\infty }d\la \,\la ^{(2p-1)\nub +2}K_{\nub +1}(a\la )<\infty , 
 && \int _{0}^{\infty }d\la \,\la ^{(2p-1)\nub +3}K_{\nub }(a\la )<\infty ,
\end{align*}
which are the case in view of \eqref{;genrel}; indeed, for the former, 
\begin{align*}
 (2p-1)\nub +2-|\nub +1|+1
 &=\frac{1}{q}\!\left( 
 -2p+1+3q-|1-q|
 \right) \\
 &>\frac{1}{q}\!\left( 
 q+1-|q-1|
 \right) >0,
\end{align*}
as well as for the latter, 
\begin{align*}
 (2p-1)\nub +3-|\nub |+1
 &=\frac{2}{q}(-p+2q)>0,
\end{align*}
by the definition~\eqref{;defnub} of $\nub $ and the assumption that $0<p<q$. 
This ends the proof of the lemma.
\end{proof}

With \lref{;lpdej} at disposal, the proof of \pref{;ph} proceeds along the same 
lines as in that of \pref{;pp}. 

\begin{proof}[Proof of \pref{;ph}]
Set the additive functional $\{ A_{t}\} _{0\le t\le 1}$ by 
\begin{align*}
 A_{t}=\int _{0}^{t}M_{s}^{q-2}|\dv _{s}|^{2}\,ds, 
\end{align*}
and apply It\^o's formula to the process 
$\{ J(A_{t},M_{t})\} _{0\le t\le 1}$ to see that, in differential notation,  
\begin{align*}
 &dJ(A_{t},M_{t})\\
 &=\frac{\partial J}{\partial s}(A_{t},M_{t})M_{t}^{q-2}|\dv _{t}|^{2}\,dt
 +\frac{\partial J}{\partial x}(A_{t},M_{t})\dv _{t}\cdot d\br _{t}
 +\frac{1}{2}\frac{\partial ^{2}J}{\partial x^{2}}(A_{t},M_{t})|\dv _{t}|^{2}\,dt\\
 &=\frac{\partial J}{\partial x}(A_{t},M_{t})\dv _{t}\cdot d\br _{t}
 +M_{t}^{p-2}|\dv _{t}|^{2}\,dt
\end{align*}
with initial value $J(A_{0},M_{0})=0$ a.s., thanks to \lref{;lpdej}. 
By taking the expectation, the assertion of the proposition 
follows readily.
\end{proof}

\subsection{Proof of \lref{;lh}}\label{;ssprflh}

In this subsection, we prove \lref{;lh} to complete the proof of \tref{;th}.

\begin{proof}[Proof of \lref{;lh}]
Fix $s,x\ge 0$ arbitrarily. By the definition~\eqref{;defj} of $J$, we see that 
$J(s,x)$ is dominated by 
\begin{align}\label{;prflhq}
 C_{p,\nub }s^{-p\nub }\!\int _{0}^{\infty }d\la \,
 \frac{\la ^{2(p-1)\nub +1}}{1+\frac{q^{2}\la ^{2}}{8}}\int _{0}^{\infty }\frac{dv}{v}\,
 v^{\nub }\exp \!\left( 
 -\frac{\la ^{2}}{4v}
 \right) ,
\end{align}
owing to Fubini's theorem. The integral with respect to $v$ is computed as 
\begin{align*}
 \int _{0}^{\infty }\frac{du}{u}\left( 
 \frac{\la ^{2}}{4u}
 \right) ^{\nub }\!e^{-u}=\frac{\Gamma (-\nub )}{2^{2\nub }}\la ^{2\nub }
\end{align*}
by changing the variables with $v=\la ^{2}/(4u)$. Plugging the last expression, 
we see that \eqref{;prflhq} turns into 
\begin{align*}
 &\frac{\Gamma (-\nub )}{2^{2\nub }}C_{p,\nub }s^{-p\nub }
 \int _{0}^{\infty }d\la \,\frac{\la ^{2p\nub +1}}{1+\frac{q^{2}\la ^{2}}{8}}\\
 &=\frac{\Gamma (-\nub )}{2^{2\nub }}C_{p,\nub }\!\left( 
 \frac{2\sqrt{2}}{q}
 \right) ^{2p\nub +1}\!\frac{\sqrt{2}}{q}s^{-p\nub }\int _{0}^{\infty }d\eta \,
 \frac{\eta ^{p\nub }}{1+\eta }\\
 &=2^{2(p-1)\nub +1}C_{p,\nub }\!\left( 
 \frac{2}{q^{2}}
 \right) ^{p\nub +1}\!\Gamma (-\nub )\Gamma (1+p\nub )\Gamma (-p\nub )s^{-p\nub },
\end{align*}
where we have changed the variables with $q^{2}\la ^{2}/8=\eta $ for the 
second line and used an alternative integral representation 
(see \cite[Equation~\thetag{1.5.3}]{leb}) of the beta function for the third. 
By the definitions~\eqref{;defnub} and \eqref{;cpnub} of $\nub $ and $C_{p,\nub }$, 
respectively, the last expression is seen to agree with the right-hand side of the 
claimed inequality~\eqref{;lhq}.
\end{proof}

Since the process $\{ M_{t}^{p-q}\} _{0\le t\le 1}$ is a submartingale, we 
see that, for every $0\le t\le 1$, 
\begin{align}\label{;bdvcons}
 \int _{0}^{t}ds\,\ex \!\left[ 
 M_{s}^{p-2}|\dv _{s}|^{2}
 \right] \le 
 \ex \!\left[ 
 M_{t}^{p-q}\int _{0}^{t}ds\,M_{s}^{q-2}|\dv _{s}|^{2}
 \right] .
\end{align}
Indeed, by Fubini's theorem, the right-hand side is equal to 
\begin{align*}
 \int _{0}^{t}ds\,\ex \!\left[ 
 M_{t}^{p-q}M_{s}^{q-2}|\dv _{s}|^{2}
 \right] 
 =\int _{0}^{t}ds\,\ex \!\left[ 
 M_{s}^{q-2}|\dv _{s}|^{2}\ex \!\left[ 
 M_{t}^{p-q}\rmid| \F _{s}
 \right] 
 \right] ,
\end{align*}
which is not less than the left-hand side. We end this section with 
a remark that the bound~\eqref{;bdv} on the function $V$ as 
observed in \rref{;rbdv} is consistent with \eqref{;bdvcons}.

\begin{rem}\label{;rcons}
By \eqref{;bdv} together with relation~\eqref{;rjv}, we have another 
bound on the function $J$: 
\begin{align*}
 J(s,x)\le x^{p-q}s
\end{align*}
for all $s\ge 0$ and $x>0$, which implies \eqref{;bdvcons} by 
virtue of \pref{;ph}.
\end{rem}

\section{Concluding remarks}\label{;scr}
In this paper, we have seen how a stochastic argument works 
effectively in the analysis of some functional inequalities. We conclude 
this paper with another instance of the effectiveness of the stochastic 
method we have employed. 

Pick $f\in \ccp $ arbitrarily. We retain the same notation as in \sref{;spre}. 
Define the $\R ^{\D \times \D }$-valued process 
$a=\bigl\{ a_{t}=(a^{ij}_{t})_{i,j=1}^{\D }\bigr\} _{0\le t\le 1}$ by 
\begin{align*}
 a^{ij}_{t}=\ex \!\left[ 
 \partial _{j}\partial _{i}f(\br _{1-t}+x)
 \right] \!\big| _{x=\br _{t}},\quad i,j=1,\ldots ,\D .
\end{align*}
We write $\ba ^{i}_{t}=(a^{ij}_{t})_{j=1}^{\D },\,0\le t\le 1$, for $i=1,\ldots ,\D $ 
below. The same as $\dv $, each component of $a$ 
is a continuous $\{ \F _{t}\} $-martingale as well. 
For every $i,j=1,\ldots ,\D $, by the boundedness of 
$\partial _{j}\partial _{i}f$, we see that 
\begin{align*}
 \ex \!\left[ 
 \partial _{j}\partial _{i}f(\br _{1-t}+x)
 \right] &=\partial _{j}\partial _{i}
 \ex \!\left[ 
 f(\br _{1-t}+x)
 \right] 
\end{align*}
for any $0\le t\le 1$ and $x\in \R ^{\D }$.
Then, in the same way as in relation~\eqref{;reldv}, the two processes 
$\dv $ and $a$ are related via 
\begin{align}
 \dv ^{i}_{t}=\ex \!\left[ \partial _{i}f(\br _{1})\right] 
 +\int _{0}^{t}\ba ^{i}_{s}\cdot d\br _{s},\quad &0\le t\le 1,\ i=1,\ldots ,\D ,\label{;rela}
\end{align}
$\pr $-a.s. It is readily seen from \eqref{;reldv} and \eqref{;rela} that 
\begin{align}\label{;qrpoi}
 \frac{d}{dt}\ex \!\left[ 
 M_{t}^{2}
 \right] &=\ex \!\left[ 
 |\dv _{t}|^{2}
 \right] , && 
 \frac{d}{dt}\ex \!\left[ 
 |\dv _{t}|^{2}
 \right] 
 =\sum _{i=1}^{\D }\ex \!\left[ 
 |\ba ^{i}_{t}|^{2}
 \right] 
 \equiv \sum _{i,j=1}^{\D }\ex \!\left[ 
 (a^{ij}_{t})^{2}
 \right] , 
\end{align}
for $0\le t\le 1$. For every $i,j=1,\ldots ,\D $, since the process 
$a^{ij}=\{ a^{ij}_{t}\} _{0\le t\le 1}$ is a martingale, its squared process 
is a submartingale. Therefore the function $\al $ 
defined by 
\begin{align}\label{;defal}
 \al (t)=\sum _{i,j=1}^{\D }\ex \!\left[ 
 (a^{ij}_{t})^{2}
 \right] ,\quad 0\le t\le 1,
\end{align}
is nondecreasing. 
Consequently, by the latter relation in \eqref{;qrpoi}, 
the function $\ex \!\left[ 
|\dv _{t}|^{2}
\right] ,\,0\le t\le 1$, is convex, from which we have 
\begin{align*}
 \ex \!\left[ 
 |\dv _{t}|^{2}
 \right] \le t\ex \!\left[ 
 |\dv _{1}|^{2}
 \right] +(1-t)|\dv _{0}|^{2}
\end{align*}
for all $0\le t\le 1$, and hence 
\begin{align*}
 \int _{0}^{1}\ex \!\left[ 
 |\dv _{t}|^{2}
 \right] dt
 \le \frac{1}{2}\left( 
 \ex \!\left[ 
 |\dv _{1}|^{2}
 \right] +|\dv _{0}|^{2}
 \right) 
\end{align*}
whose graphical interpretation is also possible. 
Combining the last inequality with the former relation in \eqref{;qrpoi} entails that 
\begin{align}\label{;crq1}
 \ex \!\left[ 
 M_{1}^{2}
 \right] -M_{0}^{2}\le \frac{1}{2}\left( 
 \ex \!\left[ 
 |\dv _{1}|^{2}
 \right] +|\dv _{0}|^{2}
 \right) ,
\end{align}
that is, 
\begin{align*}
 I_{2}(f)
 \le \frac{1}{2}\norm{|\nabla f|}{2}^{2}
 +\frac{1}{2}\left| 
 \int _{\R ^{\D }}\nabla f\,d\gss{\D }
 \right| ^{2},
\end{align*}
which is an improved Poincar\'e's inequality shown in 
\cite[Theorem~A.2]{gnp} based on the spectral decomposition of 
the Ornstein--Uhlenbeck operator in $\R ^{\D }$. Applying the 
same argument to the function $\al $, we see that it is also convex. 
Therefore, by the latter relation in \eqref{;qrpoi}, we have 
\begin{align*}
 \ex \!\left[ 
 |\dv _{t}|^{2}
 \right] -|\dv _{0}|^{2}
 \le \frac{1}{2}t\!\left( 
 \al (0)+\al (t)
 \right) 
\end{align*} 
for all $0\le t\le 1$. Noting that, thanks to the latter relation in \eqref{;qrpoi} 
and by the integration by parts formula,
\begin{align*}
 \int _{0}^{1}dt\,t\al (t)
 =\ex \!\left[ 
 |\dv _{1}|^{2}
 \right] -\int _{0}^{1}dt\,\ex \!\left[ 
 |\dv _{t}|^{2}
 \right] ,
\end{align*}
we integrate both sides of the last inequality with respect to $t$ over $[0,1]$ to get 
\begin{align*}
 \int _{0}^{1}dt\,\ex \!\left[ 
 |\dv _{t}|^{2}
 \right] -|\dv _{0}|^{2}
 \le \frac{1}{4}\al (0)
 +\frac{1}{2}\ex \!\left[ 
 |\dv _{1}|^{2}
 \right] -\frac{1}{2}\int _{0}^{1}dt\,\ex \!\left[ 
 |\dv _{t}|^{2}
 \right] .
\end{align*}
Rearranging terms and applying the former relation in \eqref{;qrpoi}, 
we arrive at a further improvement of Poincar\'e's inequality 
involving the second derivatives of $f$ as follows:
\begin{align}\label{;crq2}
 \ex \!\left[ 
 M_{1}^{2}
 \right] -M_{0}^{2}\le \frac{1}{3}
 \ex \!\left[ 
 |\dv _{1}|^{2}
 \right] +\frac{2}{3}|\dv _{0}|^{2}
 +\frac{1}{6}\al (0), 
\end{align}
namely, by the definition~\eqref{;defal} of $\al $,
\begin{align*}
 I_{2}(f)\le \frac{1}{3}\norm{|\nabla f|}{2}^{2}
 +\frac{2}{3}\left| 
 \int _{\R ^{\D }}\nabla f\,d\gss{\D }
 \right| ^{2}
 +\frac{1}{6}\sum _{i,j=1}^{\D }\left( 
 \int _{\R ^{\D }}\partial_{j}\partial _{i}f\,d\gd 
 \right) ^{2}.
\end{align*}
Inequality~\eqref{;crq1} is recovered by inserting the bound 
\begin{align*}
 \al (0)\le 
 \ex \!\left[ 
 |\dv _{1}|^{2}
 \right] -|\dv _{0}|^{2}
\end{align*}
into the right-hand side of \eqref{;crq2}. The last inequality is 
a consequence of the monotonicity of the function $\al $ 
together with the latter relation in \eqref{;qrpoi}.


\end{document}